\documentclass[12pt]{amsart}

\usepackage{amscd}
\usepackage{amssymb}
\usepackage{amsmath}

\newtheorem{theorem}{Theorem}[section]
\newtheorem{corollary}[theorem]{Corollary}
\newtheorem{lemma}[theorem]{Lemma}
\newtheorem{proposition}[theorem]{Proposition}

\newtheorem{definition}[theorem]{Definition}

\newtheorem{remark}[theorem]{Remark}
\def \nmatrix #1 #2 #3 #4 {
#1=\begin{pmatrix}
 0 & #3 & #4 \\
 0 & 0  & #2 \\
 0 & 0  & 0
\end{pmatrix}
}
\def \End {\operatorname{End}}
\def\YM {\operatorname{YM}}

\def \MZ {\mathbb{R}\times\mathbb{T}\times\mathbb{Z}}
\def \M  {\mathbb{R}\times\mathbb{T}}
\def \Q {D^{c}_{\mu \,\nu}}
\def \E {E^{c}_{\mu\, \nu}}
\def \dx {\dfrac{\partial}{\partial x}}
\def \dy {\dfrac{\partial}{\partial y}}

\begin{document}

\title [On a minimum of Yang-Mills  on  Heisenberg manifolds]
       {On a minimum of Yang-Mills functional on quantum Heisenberg manifolds}

\begin{abstract}
 In this paper, we study the Yang-Mills functional on quantum Heisenberg manifolds using the appratuses developed by A. Connes and M. Rieffel. It is discovered that a connection on a projective module over a quantum Heisenberg manifold is a minimum of Yang-Mills functional whicih is a critical point that is different with critical points found by S. Kang.
\end{abstract}

\author { Hyun Ho \quad Lee }
\address {Department of Mathematical Science\\
         Seoul National University\\
         Seoul, Korea 151-747 }
\email{hyun.lee7425@gmail.com} 
\date{ Dec 15, 2009}
\maketitle

\section {Introduction}
Classical Yang-Mills theory is concerned with the set of connections
(i.e. gauge potentials) on a vector
 bundle of a smooth manifold. The Yang-Mills functional $\YM$ measures the
 ``strength'' of the curvature of a connection. The Yang-Mills
 problem is determining the nature of the set of connections
 where $\YM$ attains its minimum, or more generally the nature of the set of
 critical points for $\YM$.  There is a well-developed non-commutative analogue of
 this setting  \cite{co1} so that we can define a non-commutative Yang-Mills problem on a $C\sp{*}$-algebra as follows.

 Let $(A,G,\alpha)$ be a $C\sp{*}$-dynamical system, where $G$ is a
 Lie group. It is said that $x$ in A is $C^{\infty}$-vector if and only
 if $g \to \alpha_g(x)$ from $G$ to the normed space is of
 $C^{\infty}$. Then $A^{\infty}=\{a\in A | \quad\text{$a$ is of $C^{\infty}$
 }\}$ is norm dense in $A$. In this case we call $A^{\infty}$ the
 smooth dense subalgebra of $A$. Since a $C\sp{*}$-algebra with a smooth
 dense subalgebra is an analogue of a smooth manifold, finitely generated
 projective $A\sp{\infty}$-modules are the appropriate generalizations of vector bundles over the
 manifold.

 \begin{lemma}\cite[Lemma 1]{co1}
 For every finite projective $A$-module $\Xi$, there exists a finite projective $A^{\infty}$-module $\Xi^{\infty}$, unique up to isomorphism,
 such that $\Xi$ is isomorphic to $\Xi^{\infty}\bigotimes_{A^{\infty}}A$.
 \end{lemma}

 In addition, an hermitian structure on $ \Xi\sp{\infty}$ is
 given by a $A\sp{\infty}$-valued positive definite inner product $<\cdot,\cdot>$ such that $<\xi, \eta>\sp{*}=<\eta, \xi>, <\xi,\eta a>=<\xi,\eta>a$  for $\xi, \eta \in \Xi\sp{\infty}$ and $a\in A\sp{\infty}$.

 If we let $\mathfrak{g}$ be the Lie-algebra of $G$, it plays the role of the tangent space of $A^{\infty}$ so that we can have `` directional derivatives'' on $A^{\infty}$.  Indeed, an infinitesimal form of $\alpha$  gives an action, $\delta$, of the Lie-algebra $\mathfrak{g}$ as follows.
 $$\delta_X(a)=\lim_{t\to 0}\dfrac{1}{t}(\alpha_{g_t}(a)-a)$$ for $X \in \mathfrak{g}$ and $a \in A^{\infty}$ where $g_t$ is the path in $G$ such that $\dot{g}_0=X$ \cite{co1}. Note that $\delta$ is the representation of $\mathfrak{g}$ in the Lie-algebra of (unbounded) derivations of $A\sp{\infty}$ preserving Lie-algebra structures.
 \begin{definition}\cite[Definition 2]{co1}\label{D:Connection}
 Given $\Xi\sp{\infty}$, a connection on $\Xi\sp{\infty}$ is a linear
 map $\nabla: \Xi\sp{\infty} \to \Xi\sp{\infty}\otimes \mathfrak{g}\sp{*} $
 such that, for all $X \in \mathfrak{g}$, $\xi \in \Xi\sp{\infty}$ and $a \in A\sp{\infty}$ one has
 \begin{equation}\label{E:connection}
 \nabla_X(\xi \cdot a)= \nabla_X(\xi)\cdot a + \xi \cdot \delta_X(a).
 \end{equation}
\end{definition}
We shall say that $\nabla$ is compatible with the hermitian metric
if and only if
\begin{equation}\label{E:compatibility}
<\nabla_X(\xi),\eta>+<\xi,\nabla_X(\eta)>=\delta_X(<\xi,\eta>)
\end{equation}
for all $\xi, \eta \in \Xi\sp{\infty}$, $X \in \mathfrak{g}$.
\begin{definition}\cite[Definition3]{co1}
Let $\nabla$ be a connection on $\Xi\sp{\infty}$, the curvature of
$\nabla$ is the element $\Theta$ of
$\End_{A\sp{\infty}}(\Xi\sp{\infty})\otimes\Lambda^2(\mathfrak{g})^*$
given by
$$\Theta_{\nabla}(X,Y)=\nabla_X\nabla_Y-\nabla_Y\nabla_X-\nabla_{[X,Y]}
$$ for all $X,Y \in \mathfrak{g}$.
\end{definition}
If $\nabla$ is compatible with the Hermitian metric, then the values
of $\Theta$ are in $E_s$, the set of skew-adjoint elements of
$E=\End_{A\sp{\infty}}(\Xi\sp{\infty})$. Since $\mathfrak{g}$ is playing the
role of the tangent space of $A\sp{\infty}$, the analogue of a
Riemannian metric on a manifold will be just an ordinary positive
inner product on $\mathfrak{g}$. With the curvature form in mind, we need the
bilinear form on the space of alternating 2-forms with values in
$E$. Then given alternating $E$-valued 2-forms $\Phi$ and $\Psi$ we
let $$\{\Phi,\Psi \}_E=\sum_{i<j}\Phi(Z_i,Z_j)\Psi(Z_i,Z_j),$$ which
is an element of $E$ where $Z_1, \cdots, Z_n$ is  an orthonormal
basis of $\mathfrak{g}$. Finally, we need an analogue of integration over a
manifold, and we need this to be $G$-invariant. Thus it is
appropriate to assume that $A\sp{\infty}$ is given a faithful trace
$\tau$ on $A\sp{\infty}$ which is invariant under the action of $\mathfrak{g}$
i.e. $\delta$-invariant so that $\tau(\delta_X(a))=0$ for all $X \in
\mathfrak{g}$ and $a\in A\sp{\infty}$. One can define an $E$-valued
inner product, $<\,,\,>_E$, by $$
<\xi,\eta>_E(\zeta)=\xi\cdot<\eta,\zeta>$$ Then every element of $E$
will be a finite linear combination of terms of form $<\xi,\eta>_E$
so that we can define a faithful trace $\tau_E$ on $E$ by
$$\tau_E(<\xi,\eta>_E)=\tau(<\xi,\eta>_{A\sp{\infty}})
$$
\begin{definition}\cite[p241]{corie}\label{D:YM1}
The Yang-Mills functional $YM$ is defined for a compatible
connection $\nabla$ by
$$ YM(\nabla)=-\tau_E(\{\Theta_\nabla,\Theta_\nabla\})$$
\end{definition}
It is not hard to show that the set of compatible connection
$CC(\Xi\sp{\infty})$ is closed under conjugation of a unitary
element of $E$. In fact, if we define $(\gamma_u(\nabla))_X
\xi=u(\nabla_X(u^*(\xi)))$ for $u \in UE$, it is easily verified
that $\gamma_u(\nabla)\in CC(\Xi\sp{\infty})$. Also, it is verified
that
$$ \Theta_{\gamma_u(\nabla)}(X,Y)=u\Theta_{\nabla}(X,Y) u^*$$ for
$X,Y \in \mathfrak{g}$, and that
$$\{\Theta_{\gamma_u(\nabla)},\Theta_{\gamma_u(\nabla)}\}=u\{\Theta_{\nabla},\Theta_{\nabla}\}u\sp{*}. $$
It follows that $$\YM(\gamma_u(\nabla))=\YM(\nabla) $$ for every $u
\in UE$ and $\nabla \in CC(\Xi\sp{\infty})$. Thus $\YM$ is a
well-defined functional on the quotient space
$CC(\Xi\sp{\infty})/UE$. If $MC(\Xi\sp{\infty})$ denotes the set of
compatible connections where $\YM$ attains its minimum, we call
$MC(\Xi\sp{\infty})/UE$ the moduli space for $\Xi\sp{\infty}$,
or more generally $\{\mbox{the set of critical points}\}/UE$ the moduli space\cite[p242]{corie}.\\

Connes and Rieffel  studied Yang-Mills for the
irrational rotation $C\sp{*}$-algebras which are non-commutative
analogue of 2-tori or non-commutative 2-tori \cite{corie} and Rieffel extended
$YM$ for the higher dimensional non-commutative tori \cite{rie}. In
view of deformation quantization, the higher dimensional
non-commutative torus is the example of deformation quantization of
$d$-dimensional torus $\mathbb{T}^d$. A further aspect of this
special deformation quantization is that the ordinary torus acts on
the non-commutative tori as a group of symmetries which is a
Lie-group action. According to Rieffel \cite{rie1}, the classical
mechanical systems which are studied possess a Lie group action of
symmetries acting on the system, and one seeks deformations which
are compatible with this Lie group action. Rieffel showed there is a
deformation quantization of the classical Heisenberg manifold,
namely, non-commutative Heisenberg manifold where the Heisenberg
group action is invariant and provided another example of a
non-commutative differential manifold \cite{rie1}. In
this short article, we investigate the Yang-Mills theory on the
non-commutative Heisenberg manifold and construct a connection which gives rise to a critical point of the Yang-Mills functional.
Moreover, we show that the connection is a minimum of the Yang-Mills functional  so that it constrasts with a certain family of connections found by S. Kang in \cite{kang}. \\

\textbf{Acknowledgements} The author learned this research, which is
still open to many questions, from Sooran Kang at GPOTS 2008 where she announced her results.  He is very grateful to her for giving him references to start with and for
answering many basic questions. Her approach to Yang-Mills on
non-commutative Heisenberg manifolds can be found in arXiv:0904.4291 or in \cite{kang}.
He also would like to thank Larry Brown for some advice.
\section{Main results}

For any positive integer $c$ let $S^{c}$ be the space of
$C^{\infty}$ functions $\phi$ on
$\mathbb{R}\times\mathbb{T}\times\mathbb{Z}$ which satisfy
\begin{itemize}
\item [(a)]$\phi(x+k,y,p)=e(ckpy)\phi(x,y,z)$ for all $k \in \mathbb{Z}$
\item [(b)]$\sup_K \|p^k\frac{\partial^{m+n}}{\partial x^m \partial y^n}\phi(x,y,p)
\| < \infty $ for all $k,m,n \in \mathbb{N}$ and  any compact set
$K$ of $\mathbb{R}\times\mathbb{T}$
\end{itemize}
We can give $S^{c}$ a $C\sp{*}$-algebra structure for each $\hbar
\in \mathbb{R}$ as follows;
\begin{itemize}
\item [(c)] $\phi \ast \psi
(x,y,p)=\sum_{q}\phi(x-\hbar(q-p)\mu,y-\hbar(q-p)\nu,q)\psi(x-\hbar
q\mu,y-\hbar q\nu,p-q)$
\item [(d)] $\phi^{\ast}(x,y,p)=\overline{\phi}(x,y,-p)$
\end{itemize}
with the norm coming from the representation on
$L^2(\mathbb{R}\times\mathbb{T}\times\mathbb{Z})$ defined by
$$\phi(f)(x,y,p)=\sum_{q}\phi(x-\hbar(q-2p)\mu,
y-\hbar(q-2p)\nu,q)f(x,y,p-q)$$ where $\mu,\nu$ are non-zero real
numbers.
 $D^{c \,\hbar}_{\mu \,\nu}$ will denote the norm completion of $S^{c}$.
 Let $G$ be the Heisenberg group given by $$(r,s,t) \leftrightarrow \begin{pmatrix}
 0 & s & t/c \\
 0 & 0  & r \\
 0 & 0  & 0
\end{pmatrix}$$
 so that when it is identified with $\mathbb{R}^3$ the product is given by  $(r,s,t)(r',s',t')=(r+r',s+s',t+t'+csr')$. Then we have a
cannonical action of
 G on $D^{c \,\hbar}_{\mu \,\nu}$ by
 $$\alpha_{(r,s,t)}(\phi)(x,y,p)=e(p(t+cs(x-r)))\phi(x-r,y-s,p)$$
 which comes from a left action of $G$ on the Heisenberg manifold and
 $(D^{c \,\hbar}_{\mu \,\nu},G,\alpha)$ is a $C\sp{*}$-dynamical system
 \cite[p557]{rie1}.

 From now on we only consider the case $\hbar=1$, and thus
 $D^{c}_{\mu \,\nu}$ will denote the corresponding $C\sp{*}$-algebra named by the quantum Heisenberg manifold.
 While the above definition was exploited for the strict deformation quantization of the Heisenberg manifold by Rieffel, we prefer the following description of $D^{c}_{\mu
 \,\nu}$.
 \begin{theorem}\cite[p17]{ab1}
 $D^{c}_{\mu \,\nu}$ is the closure in the multiplier algebra of
 $C_{0}(\mathbb{R}\times\mathbb{T})\times_{\lambda}\mathbb{Z}$ of
 the $*$-subalgebra $D_0$ consisting of functions $\phi$ in $C(\MZ)$
 which have compact support on $\mathbb{Z}$ and satisfy
 \begin{equation*}\label{E:1}
 \phi(x+k,y,p)=e(-ckp(y-p\nu))\phi(x,y,p)
 \end{equation*}
 for all $k,p \in
 \mathbb{Z}$, and $(x,y) \in \M$ where $\lambda(x,y)=(x+2\mu,y+2\nu)$
 \end{theorem}
 \begin{remark}
 In fact, $D^{c}_{\mu \,\nu}$ is the generalized fixed point algebra of $C_{0}(\mathbb{R}\times\mathbb{T})\times_{\lambda}\mathbb{Z}$
 under the action $\rho$ of $\mathbb{Z}$ given by $(\rho_k \phi)(x,y,p)=e(ckp(y-p\nu))\phi(x+k,y,p)$ \cite[Proposition 2.8]{ab1}. Under the map $J$
 given in \cite[p547]{rie1}, it is identified with the previous $D^{c}_{\mu \, \nu}$.
 \end{remark}
 As is shown in \cite{ab2}, there is a faithful trace $\tau_{D}$ on $\Q$ defined for $\phi \in
 \Q$, by
 \begin{equation}\label{E:trace}
 \tau_{D}(\phi)=\int_{\mathbb{T}^2}\phi(x,y,0)dxdy
 \end{equation}

The virtue of this definition $D^{c}_{\mu \,\nu}$ is the much simpler Morita
equivalence picture than the original Morita equivalence found by
Rieffel(For the latter, see \cite[Theorem 5.5]{rie2}).
\begin{theorem}\cite[Theorem 2.12]{ab1}\label{T:morita}
 Let $E^{c}_{\mu \,\nu}$  be the closure in the multiplier algebra of
 $C(\M)\times_{\sigma}\times\mathbb{Z}$ of the $\ast$-subalgebra
 $E_0$ consisting of functions $\psi$ in $C(\MZ)$ with compact
 support on $\mathbb{Z}$ and satisfy
 $$\psi(x-2p\mu,y-2p\nu,k)=e(ckp(y-p\nu))\psi(x,y,k)$$ for all $k,p \in
 \mathbb{Z}$, and $(x,y)\in \M$ where $\sigma(x,y)=(x-1,y)$. Then $E^{c}_{\mu
 \,\nu}$ and $D^{c}_{\mu \,\nu}$ are strong-Morita equivalent.
\end{theorem}
\begin{remark}
The bi-module $\Xi$ is the completion of $C_{c}(\mathbb{R}\times\mathbb{T})$ with respect to either one of
the norms induced by $_{D}< \cdot,\cdot>$ and $<\cdot,\cdot>_E$ respectively, given by
\begin{equation*}\label{E:rightinner}
_{D^{c}_{\mu\,\nu}}<f,g>(x,y,p)=\sum_k
e(ckp(y-p\nu))f(x+k,y)\overline{g}(x-2p\mu+k,y-2p\nu)
\end{equation*}
\begin{equation*}\label{E:leftinner}
<f,g>_{E^{c}_{\mu\,\nu}}(x,y,k)=\sum_p e(-ckp(y-p\nu))\overline{f}(x-2p\mu,y-2p\nu)g(x-2p\mu+k,y-2p\nu)
\end{equation*}
for $f,g \in C_{c}(\mathbb{R}\times\mathbb{T})$\cite{ab1}.\\
\end{remark}

Since we follow the original approach of Abadie, $\Xi$ is a left-$\Q$ and right-$\E$ bi-module. However, recall that the original definitions of connection, curvature, Yang-Mills functional, etc. were stated in the context of the right Hilbert module over a $C\sp{*}$-algebra. Thus, to define Yang-Mills problem for the quantum Heisenberg manifold, we need to work with the left-$\E$ and right-$\Q$ bi-module. This can be done using the dual module $\widetilde \Xi$. In fact, if $X$ is an $A-B$-bimodule, i.e. a left $A$ and right-$B$ module, let $\widetilde X$ be the conjugate vector space, so that there is an additive map $\flat:X \to \widetilde X$ such that $\flat(\lambda \cdot x)=\bar{\lambda}\cdot\flat(x)$. Then $\widetilde X$ is a $B-A$-bimodule with
\[
\begin{aligned}
b\cdot\flat(x)&=\flat(x \cdot b\sp{*})   \\
_B<\flat(x),\flat(y)> &=<x,y>_B \quad &
 \end{aligned}
 \begin{aligned}
 \flat(x)\cdot a &=\flat(a\sp{*}\cdot x)\\
 <\flat(x),\flat(y)>_A &=_A<x,y>
 \end{aligned}
 \]
 for $x,y \in X$, $a\in A$, and $b \in B$.
 Applying the fact just explained to $\Xi$, we obtain a left-$\E$ and right-$\Q$ bi-module $\widetilde \Xi$ from $\Xi$. But in practice we will not distinguish $\widetilde \Xi$ and $\Xi$ so that we may not use $\flat$.
\begin{corollary}\label{C:Morita}
$\Xi$ is a finitely generated $\Q$-module and $\End_{\Q}(\Xi)$ is
isomorphic to $E^{c}_{\mu\,\nu}$ via the map $f \to f\cdot\psi\sp{*}$ for $\psi \in \E$.
\end{corollary}
\begin{proof}
 Note that both $E^{c}_{\mu \,\nu}$ and $\Q$ have identity elements \cite[p309]{ab2}. From the strong Morita equivalence,
 it is well known that $\Xi$ is a finitely generated and the endomorphism ring $\End_{\Q}(\Xi)$ is equal to
$E^{c}_{\mu \, \nu}$(see \cite[Proposition 2.1]{rie2}).
\end{proof}
For our purpose, we note the right action of $\Q$ on $\Xi$ is given by
\begin{equation*}
(f\cdot
\phi)(x,y)=\sum_p\phi\sp{*}(x,y,p)f(x-2p\mu,y-2p\nu)=\sum_p\overline{\phi}(x+2p\mu,y+2p\nu,p)f(x+2p\mu,y+2p\nu)
\end{equation*}
for $\phi \in \Q$ and $f\in \Xi$\\
and $\Q$-valued inner product by
\begin{equation*}\label{E:inner}
<f,g>_{\Q}(x,y,p)=\sum_k
\bar{e}(ckp(y-p\nu))f(x+k,y)\overline{g}(x-2p\mu+k,y-2p\nu)
\end{equation*}
for $f,g \in \Xi$.\\

Under the map $J$ in \cite{rie1}, $G$ act on $\Q $ by
\begin{equation}\label{E:action}
\alpha_{(r,s,t)}(\phi)(x,y,p)=e(p(t+cs(x-p\mu-r)))\phi(x-r,y-s,p)
\end{equation}
for $\phi \in \Q$ and $(\Q, G, \alpha)$ becomes a
$C\sp{*}$-dynamical system. Also, we can check that $\tau$ is
$\delta$-invariant using (\ref{E:trace}). Since we never work with
$\Q$ and $\Xi$, but only with $C^{\infty}$ versions, we shall denote
the latter by $\Q$ and $\Xi$. It is easy to see that the Lie algebra
of the Heisenberg group $G$ has an orthonormal basis consisting of
$$\nmatrix X 1 0 0 , \nmatrix Y 0 1 0 ,\nmatrix Z 0 0 1/c .$$

Then we have the derivations corresponding to this basis.
The associated derivations $\delta_X$, $\delta_Y$, $\delta_Z$ are given by
\begin{eqnarray*}
\delta_{X}(\phi)(x,y,p)&=&-\dx
\phi(x,y,p),\\
\delta_{Y}(\phi)(x,y,p)&=&-\dy \phi(x,y,p)+2\pi i c
p(x-p\mu)\phi(x,y,p),\\
\delta_{Z}(\phi)(x,y,p)&=& 2\pi i p \,\phi(x,y,p).
\end{eqnarray*}
In our case, the exponential map from $\mathfrak{g}.$ to $G$ is given by $exp(A)=I+\sum_{n=1}A^n/n!$. Using the action of G
defined as (\ref{E:action}), we can compute
$\delta_X(\phi)=\lim_{t\to
0}\dfrac{1}{t}(\alpha_{exp(tX)}(\phi)-\phi)$ for $X \in
\mathfrak{g}$ and $\phi \in \Q$. Here we view $(r,s,t)(\in G)$ as $\left(\begin{smallmatrix}
                                                             1&s&t/c \\
                                                             0&1&r \\
                                                             0&0&1
                                                             \end{smallmatrix}\right)$.

Thus the smooth dense subalgebra of $\Q$ consisting of smooth vectors for the above derivations is the space of
complex valued functions $f$ on $\MZ$ satisfying (\ref{E:1}) such that $f$ is rapidly decreasing in the direction of $\mathbb{R}$ and $\mathbb{Z}$
and smooth with respect to $\mathbb{T}$ and the
corresponding finitely generated module is the Schwartz space of
complex valued functions on the $\M$.

Following the approach in the case of non-commutative torus \cite{corie}, it might be natural to attempt
``first order derivatives'' for the connection so that we define a linear map $\nabla$ on $\Xi$ by
\begin{subequations}\label{E:nabla}
\begin{gather}
(\nabla^0_{X}f)(x,y)=-\frac{\partial}{\partial x}f(x,y), \label{E:nabla-1}\\
(\nabla^0_{Y}f)(x,y)=-\frac{\partial}{\partial y}f(x,y)+\frac{\pi c
i}{2\mu}x^2f(x,y),\label{E:nabla-2}\\
(\nabla^0_{Z}f)(x,y)=\frac{\pi i x}{\mu}f(x,y)\label{E:nabla-3}
\end{gather}
\end{subequations}
for $f \in \Xi$ and  $X, Y, Z$ of $\mathfrak{g}$ as above.
\begin{proposition}
The linear map $\nabla^0$ defined as \eqref{E:nabla} is a compatible connection.
\end{proposition}
\begin{proof}
First, we need to verify that $\nabla^0$ satisfies the condition
(\ref{E:connection}) in Definition \ref{D:Connection}. 

It is enough to check the condition for the basis.
\begin{enumerate}
\item[(i)] $\nabla^0_X(f \cdot \phi )=\nabla^0_X(f)\cdot \phi +
 f \cdot \delta_X(\phi)$ is
follows from the Leibnitz rule.

\item[(ii)] Note that
\begin{equation*}
\begin{split}
&((\nabla^0_{Y} f)\cdot \phi)(x,y)= -\sum_p \bar{\phi}(x+2p\mu,y+2p\nu,p)
\dfrac{\partial}{\partial y}f(x+2p\mu,y+2p\nu))\\
&+ \sum_p\bar{\phi}(x+2p\mu,y+2p\nu,p)\dfrac{\pi c i}{2\mu}(x+2p\mu)^2f(x+2p\mu,y+2p\nu),\\
&( f\cdot \delta_{Y}(\phi))(x,y)= -\sum_p \dfrac{\partial}{\partial
y}\bar{\phi}(x+2p\mu,y+2p\nu,p)f(x+2p\mu,y+2p\nu)\\
&-\sum_{p}2\pi i c p(x+p\mu)\bar{\phi}(x+2p\mu,y+2p\nu,p)f(x+2p\mu,y+2p\nu). \\
\end{split}
\end{equation*}
Thus
\[
\begin{split}
&(\nabla^0_Y(f)\cdot \phi) (x,y) +
 (f \cdot \delta_Y)(\phi)(x,y) = -\sum_p \bar{\phi}(x+2p\mu,y+2p\nu,p)\\
&\dfrac{\partial}{\partial y}f(x+2p\mu,y+2p\nu)
-\sum_p \dfrac{\partial}{\partial
y}\bar{\phi}(x+2p\mu,y+2p\nu,p)f(x+2p\mu,y+2p\nu)\\
&+\sum_p\bar{\phi}(x+2p\mu,y+2p\nu,p)[\dfrac{\pi c i}{2\mu}(x+2p\mu)^2-2\pi i c p(x+p\mu)]f(x+2p\mu,y+2p\nu)\\
&=-\dfrac{\partial}{\partial y}(f \cdot \phi)(x,y)+\dfrac{\pi c i}{2\mu}x^{2}(f \cdot \phi
)(x,y)=\nabla^0_{Y}(f \cdot \phi)(x,y).
\end{split}
\]
\item [(iii)] Similarly, note that
\begin{align*}
((\nabla^0_{Z}f)\cdot \phi)(x,y)&= \sum_p \dfrac{\pi i(x+2p\mu)}{\mu}\bar{\phi}(x+2p\mu,y+2p\nu,p)f(x+2p\mu,y+2p\nu), \\
(f\cdot(\delta_{Z}(\phi))(x,y)&= -\sum_p 2\pi i p \bar{\phi}(x+2p\mu,y+2p\nu,p)f(x+2p\mu,y+2p\nu).
\end{align*}
Then
\begin{align*}
&(\nabla^0_{Z}( f \cdot \phi))(x,y)=\dfrac{\pi i x}{\mu}\sum_p \bar{\phi}(x+2p\mu,y+2p\nu,p) f(x+2p\mu,y+2p\nu)\\
&= \sum_p\left(\dfrac{\pi i (x+2p\mu)}{\mu}-2\pi i
p\right)\bar{\phi}(x+2p\mu,y+2p\nu,p)f(x+2p\mu,y+2p\nu)\\
&=((\nabla^0_{Z}f)\cdot \phi)(x,y)+(f\delta_{Z}(\phi))(x,y)
\end{align*}
\end{enumerate}
We also need to check the compatibility (\ref{E:compatibility}).

\begin{enumerate}
\item[(i)] By the product rule of differentiation, it follows
that
$\delta_{X}(<f,g>_{\Q})=<\nabla^0_{X}f,g>_{\Q}+<f,\nabla^0_{X}(g)>_{\Q}$.
\item[(ii)] Note that 
\begin{align*}
&\delta_{Y}(<f,g>)(x,y,p) =-\dy(<f,g>)(x,y,p)+2\pi i c p(x-p\mu)(<f,g>)(x,y,p)\\
&=\sum_k ((2\pi i ckp)+2\pi i c p (x-p\mu)) \bar{e}(ckp(y-p\nu))
f(x+k,y)\bar {g}(x-2p\mu+k,y-2p\nu)\\
&-\sum_k \bar{e}(ckp(y-p\nu))\dy f(x+k,y)
\bar{g}(x-2p\mu+k,y-2p\nu)\\
&-\sum_k \bar{e}(ckp(y-p\nu)) f(x+k,y)\dy
\bar{g}(x-2p\mu+k,y-2p\nu),\\
&<\nabla^0_{Y}f,g>(x,y,p)=-\sum_k \bar{e}(ckp(y-p\nu))\dy f(x+k,y)\bar{g}(x-2p\mu+k,y-2p\nu)\\
&+ \sum_k \dfrac{\pi ci}{2\mu}(x+k)^2 \bar{e}(ckp(y-p\nu))f(x+k,y)\bar{g}(x-2p\mu+k,y-2p\nu)),\\
&<f,\nabla^0_{Y}g>(x,y,p)= -\sum_k \bar{e}(ckp(y-p\nu)) f(x+k,y)\dy\bar{g}(x-2p\mu+k,y-2p\nu)\\
&-\sum_k \dfrac{\pi ci}{2\mu}(x-2p\mu+k)^2 \bar{e}(ckp(y-p\nu))f(x+k,y)\bar{g}(x-2p\mu+k,y-2p\nu)).
\end{align*}
Thus 
\begin{align*}
&<\nabla^0_{Y}f,g>(x,y,p)+<f,\nabla^0_{Y}g>(x,y,p)=\sum_k \left(\dfrac{\pi c i}{2\mu}(x+k)^2- \dfrac{\pi c
i}{2\mu}(x-2p\mu+k)^2\right)\\
&\bar{e}(ckp(y-p\nu))f(x+k,y)\bar {g}(x-2p\mu+k,y-2p\nu)\\
&-\sum_k \bar{e}(ckp(y-p\nu))\dy f(x+k,y)
\bar{g}(x-2p\mu+k,y-2p\nu)\\
&-\sum_k \bar{e}(ckp(y-p\nu)) f(x+k,y)
\dy\bar{g}(x-2p\mu+k,y-2p\nu)\\
&=\sum_k \left(2\pi icp(x-p\mu)+2\pi i cpk \right)\bar{e}(ckp(y-p\nu))f(x+k,y)\bar {g}(x-2p\mu+k,y-2p\nu)\\
&-\sum_k \bar{e}(ckp(y-p\nu))\dy f(x+k,y)
\bar{g}(x-2p\mu+k,y-2p\nu)\\
&-\sum_k \bar{e}(ckp(y-p\nu)) f(x+k,y)
\dy\bar{g}(x-2p\mu+k,y-2p\nu)\\
&=\delta_{Y}(<f,g>)(x,y,p).
\end{align*}

\item[(iii)] Note that
\[<\nabla^0_{Z}(f),g>(x,y,p)=\sum_k \bar{e}(ckp(y-p\nu))\dfrac{\pi i(x+k)}{\mu}f(x+k,y)\bar{g}(x-2p\mu+k,y-2p\nu),\]
\[<f,\nabla^0_{Z}(g)>(x,y,p)=\sum_k \bar{e}(ckp(y-p\nu)) f(x+k,y)\overline{\dfrac{\pi
i(x-2p\mu+k)}{\mu}}\bar{g}(x-2p\mu+k,y-2p\nu).\]
Hence,
\begin{align*}
\begin{split}
&<\nabla^0_{Z}(f),g>(x,y,p)+<f,\nabla^0_{Z}(g)>(x,y,p)\\
&=\sum_k \left(\dfrac{\pi i(x+k)}{\mu}\right)\bar{e}(ckp(y-p\nu))f(x+k,y)\bar{g}(x-2p\mu+k,y-2p\nu)\\
&+\left(\overline{\dfrac{\pi i(x-2p\mu+k)}{\mu}}\right)\bar{e}(ckp(y-p\nu))f(x+k,y)\bar{g}(x-2p\mu+k,y-2p\nu)\\
&=\sum_k \left(\dfrac{\pi i(x+k)}{\mu}+\overline{\dfrac{\pi i(x-2p\mu+k)}{\mu}}\right)\bar{e}(ckp(y-p\nu))f(x+k,y)\bar{g}(x-2p\mu+k, y-2p\nu)\\
&=\sum_k (2\pi i p)\bar{e}(ckp(y-p\nu))f(x+k,y)\bar{g}(x-2p\mu+k,y-2p\nu)\\
&=\delta_{Z}(<f,g>)(x,y,p).
\end{split}
\end{align*}
\end{enumerate}
\end{proof}

If we let $\{Z_{i}\}$ be the basis of a Lie algebra
$\mathfrak{g}$,  a linear map $\hat{\nabla}\sp{\ast}$ which takes
$E_s$-valued 2 forms $\Omega$ to 1-forms is defined by
\begin{equation*}
(\hat{\nabla}\sp{\ast}\Omega)(Z_i)=\sum_j[\nabla_{Z_j},\Omega(Z_i\wedge
Z_j)]-\sum_{j<k}c^{i}_{j\,k}\Omega(Z_j \wedge Z_k)
\end{equation*}
where $c^{i}_{j\,k}$ are the structure constants of $\mathfrak{g}$
for the basis ${Z_j}$.
\begin{lemma}\cite[Theorem 1.1]{rie3}\label{L:YME}
A compatible connection $\nabla$ is a critical point of $\YM$
exactly when it satisfies the Yang-Mills equation
$\hat{\nabla}\sp{\ast}(\Theta_{\nabla})=0$
\end{lemma}
\begin{theorem}\label{T:YM}
The connection $\nabla^0: \Xi \to \Xi \otimes
\mathfrak{g}\sp{*} $ satisfying (\ref{E:nabla-1}), (\ref{E:nabla-2}), (\ref{E:nabla-3}) is a critical point of
Yang-Mills equation for the quantum Heisenberg manifold $D^{c}_{\mu
\,\nu}$.
\end{theorem}
\begin{proof}
First, we note that $[X,Y]=-cZ$, $[Y,Z]=0$, $[Z,X]=0$. Therefore,
$C^{3}_{1\,2}=-c$ and all other structure constants are zero. We
will show that the Yang-Mills equation
$(\hat{\nabla^0}\sp{\ast}(\Theta_{\nabla^0}))=0$ for our choice
$\nabla^0$. Thus it is enough to show that
\[(\hat{\nabla^0}\sp{\ast}(\Theta_{\nabla^0}))(Z_i)=\sum_j[\nabla^0_{Z_j},\Theta_{\nabla^0}(Z_i\wedge
Z_j)]-\sum_{j<k}c^{i}_{j\,k}\Theta_{\nabla^0}(Z_j \wedge Z_k)=0\] for
each $i$. \\
Let  $Z_1=X$, $Z_2=Y$, $Z_3=Z$. \\
\begin{align*}
\begin{split}
(\Theta_{\nabla^0}(X,Y)f)(x,y)&= \nabla^0_X(\nabla^0_Y f)(x,y)-\nabla^0_Y(\nabla^0_X f)(x,y)-(\nabla^0_{[X,Y]}f)(x,y)\\
&=\dfrac{\partial}{\partial x \partial y}f(x,y)-(\dx\dfrac{ic\pi x^2}{2\mu})f(x,y)-(\dfrac{ic\pi x^2}{2\mu}) \dx f(x,y)\\
&-\dfrac{\partial}{\partial x \partial y}f(x,y)+ (\dfrac{ic\pi x^2}{2\mu})\dx f(x,y)+c(\nabla^0_Z f)(x,y)\\
&=(\dfrac{\pi c ix}{\mu}-\dfrac{\pi c
i x}{\mu})f(x,y)\\
&=0, \quad \text{i.e.,}\, \Theta_{\nabla^0}(X,Y)=0.
\end{split}
\end{align*}

\begin{align*}
\begin{split}
(\Theta_{\nabla^0}(Y,Z)f)(x,y)&=\nabla^0_Y (\nabla^0_Z f)(x,y)-\nabla^0_Z(\nabla^0_Y f)(x,y)-(\nabla^0_{[Y,Z]}f)(x,y)\\
&= -\dy (\nabla^0_Z f)(x,y)
+\dfrac{ic\pi x^2}{2\mu}(\nabla^0_Z f)(x,y)-\dfrac{\pi i x}{\mu}(\nabla^0_Y f)(x,y)\\
&=-\dfrac{\pi i x}{\mu}\dy f(x,y)+ \dfrac{\pi i x}{\mu}(\dfrac{ic\pi x^2}{2\mu})f(x,y)\\
& + \dfrac{\pi i x}{\mu}\dy f(x,y)- \dfrac{\pi i x}{\mu}(\dfrac{ic\pi x^2}{2\mu})f(x,y)\\
&=0,  \quad \text{i.e.,}\, \Theta_{\nabla^0}(Y,Z)=0.
\end{split}
\end{align*}

\begin{align*}
\begin{split}
(\Theta_{\nabla^0}(Z,X)f)(x,y)&=\nabla^0_Z(\nabla^0_X f)(x,y)-\nabla^0_X(\nabla^0_Z f)(x,y)-(\nabla^0_{[Z,X]}f)(x,y)\\
&=-\dfrac{\pi i x}{\mu}\dx f (x,y) + \dx ((-\dfrac{\pi i x}{\mu} f(x,y))\\
&= -\dfrac{\pi i }{\mu}f(x,y), \quad \text{i.e.,}\, \Theta_{\nabla^0}(Z,X)=\dfrac{\pi i}{\mu}I_{E}.
\end{split}
\end{align*}
where $I_E$ is the identity element of $\E=\End(\Xi)$.
Thus
\begin{align*}
(\hat{\nabla^0}\sp{\ast}\Theta_{\nabla^0})(X)&=[\nabla^0_{Y},\Theta_{\nabla^0}(X,Y)]+[\nabla^0_{Z},\Theta_{\nabla^0}(X,Z)]=0\\
(\hat{\nabla^0}\sp{\ast}\Theta_{\nabla^0})(Y)&=[\nabla^0_{X},\Theta_{\nabla^0}(Y,X)]=0\\
(\hat{\nabla^0}\sp{\ast}\Theta_{\nabla^0})(Z)&=[\nabla^0_{X},\Theta_{\nabla^0}(Z,X)]-c^{3}_{1
\, 2}\Theta_{\nabla^0}(X,Y)=0
\end{align*}
\end{proof}
\begin{remark}\label{R:curvature}
It is said that a connection $\nabla$ has ``constant curvature'' if there is a complex-valued alternating 2 form $\kappa$
such that
\[ \Theta_{\nabla}(X,Y)=\kappa(X,Y)I_E\]
for $X,Y \in \mathfrak{g}$ \cite[P243]{corie}. The proof of Theorem \ref{T:YM} shows that the connection $\nabla^0$ defined by
(\ref{E:nabla}) has constant curvature.
\end{remark}
\begin{lemma}\label{L:skew}
If $\nabla$ and $\nabla^{0}$ are two compatible connections, then $\nabla_X-\nabla^{0}_X$ is a skew-adjoint element of $E$ for every
$X \in \mathfrak{g}$.
\end{lemma}
\begin{proof}

\begin{align*}
(\nabla_X-\nabla^{0}_X)(\xi \cdot a)&=(\nabla_X \xi)\cdot a-\xi\cdot \delta_X(a) - ((\nabla^0_X \xi)\cdot a-\xi\cdot \delta_X(a))\\
                                    &=(\nabla_X \xi)\cdot a - (\nabla^0_X \xi)\cdot a \\
                                    &=(\nabla_X-\nabla^{0}_X)(\xi) \cdot a
\end{align*}
Since
\begin{align*}
\delta_{X}(<\xi,\eta>)&=<\nabla_X \xi, \eta>+ <\xi, \nabla_X \eta>\\
                      &=<\nabla^0_X \xi, \eta>+ <\xi, \nabla^0_X \eta>
\end{align*}
, we know that $<\nabla_X \xi, \eta>-<\nabla^0_X \xi, \eta>=<\xi, \nabla^0_X \eta>-<\xi, \nabla_X \eta>$.

Thus
 \begin{align*}
 <(\nabla_X-\nabla^{0}_X)^{\ast}(\xi), \eta> &=<\xi, (\nabla_X-\nabla^{0}_X)(\eta)>\\
                                              &=-(<\xi, \nabla^0_X \eta>-<\xi, \nabla_X \eta>)\\
                                              &=-(<\nabla_X \xi, \eta>-<\nabla^0_X \xi, \eta>)\\
                                              &=<-(\nabla_X-\nabla^{0}_X)\xi,\eta>
 \end{align*}
\end{proof}
\begin{corollary}
All other critical points of Yang-Mills equation are of the from
$\nabla^0 + \rho$ where $\rho(Z_i) \in E^{c}_{\mu \, \nu}$ such
that $\rho(Z_i)$ is an imaginary valued function for each $Z_i$. In
addition, $(\nabla^0_{X}+\rho_{X})(f)=\nabla^0_{X}(f)+\rho(X)\cdot f$
where the latter is defined by the action of $E^{c}_{\mu \,\nu}$ on
$\Xi$.
\end{corollary}
\begin{proof}
Suppose $\nabla$ be another critical point. Then
$\nabla_{X}-\nabla^0_{X}$ is a skew-adjoint element of
$E=\End_{\Q}(\Xi)$ for each $X \in \mathfrak{g}$ by Lemma \ref{L:skew}. If let
$\rho(X)=\nabla_{X}-\nabla^0_{X}$, ,
$\nabla_{X}=\nabla^0_{X}+\rho(X)$ and the skew-adjointness of $\rho(X) \in \E$ implies that it
it pure-imaginary valued.
 \end{proof}

 Next, we classify the nature of the critical point $\nabla^0$. First, we need a lemma.
 \begin{lemma}\cite[Lemma 2.2]{corie}\label{L:killing}
Given a connection $\nabla$, define a covariant derivative
$\hat{\delta}_X$ for each $X \in \mathfrak{g}$ by
$\hat{\delta}_X(T)=[\nabla_X,T]$ for $T \in E$. Then $\tau_E$ is
$\hat{\delta}$-invariant in the sense that
\[
 \tau_E(\hat{\delta}_X(T))=0
\]
for all $T \in E$ and $X \in \mathfrak{g}$.
\end{lemma}

\begin{theorem}\label{T:Minima}
The connection $\nabla^0: \Xi\sp{\infty} \to \Xi\sp{\infty}\otimes
\mathfrak{g}\sp{*} $ satisfying (\ref{E:nabla-1}), (\ref{E:nabla-2}), and
(\ref{E:nabla-3}) is a minimum of the Yang-Mills functional for the
quantum Heisenberg manifold $D^{c}_{\mu \,\nu}$.
\end{theorem}
\begin{proof}
As we have seen in Remark \ref{R:curvature}, $\nabla^0$ has constant curvature. Thus
$$\Theta_{\nabla^0}(Z_i,Z_j)=\kappa(Z_i,Z_j)I_E$$
for $Z_i, Z_j \in \mathfrak{g}$, where $\kappa$ is a complex valued
alternating form on $\mathfrak{g}$. This simplifies the analysis of the $\YM$.
Note that every other compatible connection is of the
form $\nabla^0+\rho$ where $\rho$ is a linear map from $\mathfrak{g}$
to the set of skew-adjoint elements of $E$ by Lemma \ref{L:skew}. If we let
$\nabla=\nabla^0+\rho$, then
\begin{equation}
\Theta_{\nabla}(X,Y)=\Theta_{\nabla^0}(X,Y)+\Omega(X,Y)
\end{equation}
where $\Omega$ is an alternating $E$-valued 2-form on $\mathfrak{g}$
defined by \[\Omega(X,Y)=
\hat{\delta}_X(\rho_Y)-\hat{\delta}_Y(\rho_X)-\rho_{[X,Y]}+[\rho_X,\rho_Y].\]
Since $\kappa$ is complex valued,
\[
 \{\Theta_{\nabla},\Theta_{\nabla}\}_E=\{
 \Theta_{\nabla^0},\Theta_{\nabla^0}\}_E+
 2\sum_{i<j}\kappa(Z_i,Z_j)\Omega(Z_i,Z_j)+\{\Omega,
 \Omega\}_E
 \]
By Lemma \ref{L:killing},
\begin{equation}\label{E:YM}
\YM(\nabla)=YM(\nabla^0)+2\sum_{i<j}\kappa(Z_i,Z_j)\tau_E(\rho_{[Z_i,Z_j]})
+\tau_E(\sum_{i<j}\Omega^*(Z_i,Z_j)\Omega(Z_i,Z_j))
\end{equation}
But $\kappa(X,Y)=0$ and $[X,Z]=[Y,Z]=0$. Thus
\begin{align*}
\YM(\nabla)&=\YM(\nabla^0)+\tau_E(\sum_{i<j}\Omega^*(Z_i,Z_j)\Omega(Z_i,Z_j))\\
            &\geq \YM(\nabla^0)
\end{align*}
It follows that $\YM$ attains its minimum at $\nabla^0$.
\end{proof}
\begin{corollary}\label{C:perturbation}
All other minimal points of Yang-Mills functional are of the form
$\nabla^0 + \rho$ where $\rho(Z_i) \in E^{c}_{\mu \, \nu}$ such
that $\rho(Z_i)$ is a imaginary valued function for each $Z_i$. In
addition, $(\nabla^0_{X}+\rho_{X})(f)=\nabla^0_{X}(f)+\rho(X)\cdot f$
where the latter is defined by the action of $E^{c}_{\mu \,\nu}$ on
$\Xi$.
\end{corollary}
\begin{proof}
This follows from Theorem \ref{T:Minima} and Corollary \ref{C:Morita}.
 \end{proof}


\begin{thebibliography}{99}

\bibitem[Ab1]{ab1} B. Abadie \emph{Generalized fixed point algebras of certain actions on crossed products}
Pacific J. Math. 171(1995), 1-21

\bibitem[Ab2]{ab2} B. Abadie \emph{``Vector bundles'' over quntum Heisenberg
manifolds} Algebraic Methods in Operator Theory, Birkh{\"{a}}user,
Boston, 1994  307-315

\bibitem[CR]{corie} A. Connes, M. Rieffel \emph{Yang-Mills for non-commutative two tori}
Contemp. Math. 62(1987) 237-266

\bibitem[Co]{co} A. Connes \emph{Noncommutative geometry},
Academic Press, San Diego, 1994

\bibitem[Co1]{co1} A. Connes \emph{$C\sp{*}$-algebres et geometrie differentiel
} C.R. Acad. Sci. Paris Ser. A-B, 290, 1980

\bibitem[Kang]{kang} S. Kang \emph{The Yang-Mills functional and Laplace's equation on quantum Heisenberg manifolds} J. Func. Anal.  257(2010), 307-327

\bibitem[Rie]{rie} M. A. Rieffel \emph{Projective modules over hinger-dimensinal non-commutative tori}
Can. J. Math. 40(1988), 257-338

\bibitem[Rie1]{rie1} M. A. Rieffel \emph{Deformation Quantization of Heisenberg manifolds}
 Comm. Math. Phys. 122(1989), 531-562

\bibitem[Rie2]{rie2}M. A. Rieffel \emph{$C\sp{*}$-algebras associated
with irrational rotations} Pacific J. Math. 93(1981), 415-429

\bibitem[Rie3]{rie3}M. A. Rieffel \emph{Critical points of Yang-Mills for non-commutative two-tori}
 J. Differential Geo. 31(1990), 535-546
\end{thebibliography}
\end{document}